\documentclass[12pt,a4paper]{article}

\usepackage{amssymb}
\usepackage{amsthm}
\usepackage{amsmath}
\usepackage[colorlinks=true,allcolors=black]{hyperref}
\usepackage{graphics}
\usepackage{enumerate}
\usepackage{tikz}
\usetikzlibrary{positioning}
\usepackage{tkz-tab}

\newtheorem{theorem}{Theorem}[section]
\newtheorem{corollary}[theorem]{Corollary}

\begin{document}

	\title{Characterizing finite groups whose order supergraphs satisfy a connectivity condition}
	\author{Ramesh Prasad Panda\thanks{Department of Mathematics, School of Advanced Sciences, VIT-AP University, Amaravati, PIN-522237, Andhra Pradesh, India.}, Papi Ray\thanks{Department of Mathematics and Statistics, Indian Institute of Technology Kanpur, Kanpur 208016, Uttar Pradesh, India}}
	
	\date{}

	\maketitle
	
	\makeatletter{\renewcommand*{\@makefnmark}{}
		\footnotetext{Email addresses: {\tt rameshprpanda@gmail.com} (R. P. Panda), {\tt popiroy93@gmail.com} (P. Ray)}
	
	\begin{abstract}
		Let $\Gamma$ be an undirected and simple graph. A set $ S $ of vertices in $\Gamma$ is called a {cyclic vertex cutset} of $\Gamma$ if $\Gamma - S$ is disconnected and has at least two components each containing a cycle. If $\Gamma$ has a cyclic vertex cutset, then it is said to be {cyclically separable}. For any finite group $G$, the order supergraph $\mathcal{S}(G)$ is the simple and undirected graph whose vertices are elements of $G$, and two vertices are adjacent if as elements of $G$ the order of one divides the order of the other. In this paper, we characterize the finite nilpotent groups and various non-nilpotent groups, such as the dihedral groups, the dicyclic groups, the EPPO groups, the symmetric groups, and the alternating groups, whose order supergraphs are cyclically separable.
		
		\vskip .5cm
		
		

	\end{abstract}

	\section{Introduction}
	
	Starting with Cayley graphs, the association of graphs with groups has a long history. These graphs were introduced by Arthur Cayley \cite{cayley1878desiderata} in 1878. Much later, in 1955, Brauer and Fowler \cite{Brauer1955} introduced the commuting graph of a group while working on the  classification of finite simple groups. Other graphs associated with groups, such as Gruenberg-Kegel graph \cite{Gruenberg1975,Williams1981}, conjugacy class graph \cite{bertram1990}, and generating graph \cite{Liebeck1996}, were defined and studied afterwards. In addition to theoretical interest, these graphs have been studied due to their different applications \cite{bertram,cooperman, Hayat2019, Kelarev2009}. The notion of the power graph of a group was introduced by Kelarev and Quinn \cite{kel-2000,kel-2002}. The \emph{power graph} of a group $ G $, denoted by $\mathcal{P}(G)$, is the undirected and simple graph with vertex set  $ G $ and two vertices are adjacent if one of them is a positive power of the other in $G$. In recent years, power graphs have been studied extensively by researchers (see \cite{AKC,chattopadhyay2019,hamzeh2017,kumar2021,panda2024,zahirovic2022} and the references therein). In particular, the invariants of power graphs, such as chromatic number \cite{xuanlong2015}, vertex connectivity \cite{panda2018a}, adjacency spectrum \cite{Mehranian2016}, minimum degree \cite{PPS-cyclic}, and automorphism group \cite{Feng2016}, are obtained. 
	
	In \cite{hamzeh2017}, Hamzeh and Ashrafi studied the automorphism groups of some supergraphs of $\mathcal{P}(G)$. One of those supergraphs, which the authors named the \emph{order supergraph} of $G$  later in \cite{hamzeh}, is the undirected and simple graph $\mathcal{S}(G)$ with vertex set  $ G $ and two vertices are adjacent if as elements of $G$ the order of one divides the order of the other. In \cite{hamzeh}, they investigated structures, parameters such independence number and number of edges, and panarity of order supergraphs of finite groups. Ma and Su \cite{ma2022ordersuper} continued the study of independence number of order supergraphs and answered a question was posed in \cite{hamzeh}. Kumar et al. \cite{Kumar2023} obtained sharp bounds for the vertex connectivity of order supergraphs of dihedral and dicyclic groups. Then Kumar et al. \cite{kumar2024} studied Hamiltonianity and provided tight bounds for vertex connectivity of order supergraphs of finite groups having an element of exponent order. In \cite{manisha}, Manisha et al. provided line graph characterization of order supergraphs of finite groups.
	
Throughout, we consider only undirected and simple graphs. A \emph{cyclic vertex cutset} of a graph $\Gamma$ is a vertex cutset $S$ of $\Gamma$ such that $\Gamma - S$ has at least two components each containing a cycle. If $\Gamma$ has a cyclic vertex cutset, then it is said to be {\em cyclically vertex separable}. The \emph{cyclic vertex connectivity} $ c\kappa(\Gamma) $ is the minimum of cardinalities of the cyclic vertex cutsets of $\Gamma$.  If $\Gamma$ has no cyclic vertex cutset, $ c\kappa(\Gamma) $ is taken as infinity. The \emph{cyclic edge connectivity} is defined by replacing vertex deletion with edge deletion. The use of the concept of cyclic connectivity dates to the work of Tait in 1880. Tait \cite{tait1880} provided a incorrect proof of the four color conjecture by using a now-disproved result which said that every planar cubic graph satisfying a cyclic connectivity condition had a Hamiltonian cycle. Later, Birkhoff \cite{birkhoff1913} showed that the four color conjecture is true for the planar graphs if it is true for the planar cubic graphs satisfying a condition involving cyclic connectivity. Appel and Haken \cite{appel1977a,appel1977b} provided a computer-assisted proof of the four color conjecture in 1977. In the last few decades, cyclic connectivity of graphs has been explored from different viewpoints, see \cite{aldred1991,liang2016,nedela2022,robertson1984} and the references therein. In fact, cyclic connectivity has been used in problems of graph theory, such as integer flow conjectures \cite{zhang1997}, $n$-extendable graphs \cite{lou1993}, and as a measure of network reliability \cite{latifi1994}. 

Before discussing the objectives of this paper, we first recall various the relevant definitions. We always denote the identity element of a group by $e$, and the order of a group element $x$ by $\circ(x)$. Given a group element $x$, we write $[x]$ for the generators of the cyclic subgroup $\langle x \rangle$. For any positive integer $n$, $\phi(n)$ denotes the Euler's phi function. We write $|A|$ for the number of elements in a set $A$. Then for any group element $x$, we have $|[x]| = \phi(\circ(x))$. For any group $G$, we write $G^* = G \setminus \{e\}$. For any positive integer $n$, we denote by $\mathbb{Z}_n$ the group of integers under addition modulo $n$. For any integer $n \geq 3$, the \emph{dihedral group} of order $2n$ is given by
$$D_{2n} = \langle a, b \mid a^n = b^2 = e, ab = ba^{-1} \rangle.$$
For any integer $n \geq 2$, the \emph{dicyclic group} of order $4n$ is given by
	$$Q_{4n} = \langle a, b \mid a^{2n} = e, a^{n} = b^2, ab = ba^{-1} \rangle.$$
	An \emph{EPPO group} is a group in which every element has prime power order. Whereas, a group is said to be \emph{EPO group} if all non-identity elements have prime orders. 
A group $G$ is said to be {\it nilpotent} if its lower central series $G=G_1 \geq G_2 \geq G_3 \geq G_4 \geq \cdots $ terminates at $\{e\}$ after a finite number of steps, where $G_{i+1}=[G_i, G]$ for $i\geq 1$. If $G$ is a finite group, then $G$ being nilpotent is equivalent to any of the following statements:
\newpage

\begin{enumerate}[\rm(a)]
\item Every Sylow subgroup of $G$ is normal.
\item $G$ is the direct product of its Sylow subgroups.
\item For $x,y\in G$, $x$ and $y$ commute whenever $\circ(x)$ and $\circ(y)$ are relatively prime.
\end{enumerate}
Recall that any abelian group and any finite $p$-group is nilpotent. Whereas, dihedral and dicyclic groups are nilpotent if and only if $n$ is a power of $2$. We finally recall that the \emph{symmetric group} $S_n$ is the group of all permutations of the the set $\{1,2,\dots, n \}$ under function composition, and that the \emph{alternating group} $A_n$ is the subgroup of $S_n$ consisting of all even permutations.

The complete graph on $n$ vertices is denoted by $K_n$. Let $\Gamma_1$ and $\Gamma_2$ be two graphs.  We write $\Gamma_1 \cong \Gamma_2$ if the two graphs are isomorphic. For pairwise disjoint graphs $\Gamma_1, \Gamma_2, \dots \Gamma_r$, we denote their union by $\Gamma_1 + \Gamma_2 + \dots + \Gamma_r$. For any graph $\Gamma$, if $\Gamma_i \cong \Gamma$ for every $i \in \{1,2,\dots, r\}$, then we denote the above disjoint union by $r\Gamma$. The \emph{join} $\Gamma_1 \vee \Gamma_2$ of two disjoint graphs $\Gamma_1$ and $\Gamma_2$ is the graph obtained by taking $\Gamma_1 + \Gamma_2$ and then adding edges $\{v_1,v_2\}$ for all vertices $v_1$ and $v_2$ in $\Gamma_1 $ and $ \Gamma_2$, respectively. Given any group $G$, we write $\mathcal{S}(A)$ for the subgraph of $\mathcal{S}(G)$ induced by $A \subseteq G$.
	
Since we are concerned with vertex deletion, instead of cyclic vertex cutset and cyclically vertex separable, we simply write cyclic cutset and cyclically separable.
	
 In \cite{panda2025pgroups}, Panda characterized the finite $p$-groups whose power graphs are cyclically separable in terms of their maximal cyclic subgroups. Then in \cite{panda2025cyclic}, Panda  determined the orders of finite cyclic, dihedral, and dicyclic groups such that their power graphs are cyclically separable.  	
By definition, the order supergraph of a finite $p$-group is complete, and so it is not cyclically separable. Whereas, it is known from \cite{hamzeh} that power graphs and order supergraphs are equal for finite cyclic groups. Thus the classification of finite cyclic groups with cyclically separable order supergraphs follows from \cite{panda2025cyclic}.  In this paper, we characterize the dihedral groups, the dicyclic groups, the EPPO groups, the finite nilpotent groups, and the symmetric and alternating groups whose order supergraphs are cyclically separable.

\section{Cyclic separability}

In this section, we investigate the existence of cyclic cutsets in order supergraphs of various finite groups. As a result, we characterize the cyclic separability of these graphs. In Theorems \ref{dihedral1} and \ref{dicyclic1}, we study the order supergraphs of dihedral and dicyclic groups, respectively, for cyclically separable condition. In these results, we follow the presentations of the respective groups given in the introduction.

	\begin{theorem}\label{dihedral1}
		For any positive integer $n \geq 3$, $\mathcal{S}(D_{2n})$ is cyclically separable if and only if the following hold:
		\begin{enumerate}[\rm(i)]
			\item $n$ is not a power of $2$,
			\item $n \geq 5$,
			\item $n \neq 6,\, 12$.
		\end{enumerate}
		
	\end{theorem}
	
	\begin{proof} Suppose that (i), (ii), and (iii) holds.
		Let $n$ be divisible by an odd number $m \geq 5$, and $x$ be an element of order $m$ in $D_{2n}$. Note that $x \in \langle a \rangle$.	
		
		\smallskip
		\noindent
		Case 1. $m$ is divisible by a prime $p \geq 5$. Then $|[x]| = \phi(m) \geq \phi(5) = 4 $. Hence $[x]$ is a clique of size at least $4$ in $\mathcal{S}(\langle a \rangle)$.
		
		\smallskip
		\noindent
		Case 2. $m$ is divisible by no primes $p \geq 5$. Then $m$ is divisible by $9$, so that $|[x]| = \phi(m) \geq \phi(9) = 6 $. Hence $[x]$ is a clique of size at least $6$ in $\mathcal{S}(\langle a \rangle)$.
		
		Let $S = \langle a \rangle {\setminus} [x]$. Then $\mathcal{S}(D_{2n}) - S $ is disconnected with two components $ \mathcal{S}([x]) $ and $ \mathcal{S}(\{b, ab, \dots, a^{n-1}b\})$.  Since each $ab^i$ has order two, $\{b, ab, \dots, a^{n-1}b\}$ is a clique of size at least $5$ in $\mathcal{S}(D_{2n})$. Whereas, as shown earlier, $[x]$ is a clique of size at least $4$. Hence $\mathcal{S}(D_{2n})$ is cyclically separable.
		
		Now, suppose that $n$ is not divisible by any odd number $m \geq 5$. As $n$ is not a power of $2$ and $n \neq 6, 12$, we have $n = 3 \cdot 2^k$ for some positive integer $k \geq 3$. Then $\langle a \rangle$ has elements, say $y$ and $z$, of order $6$ and $8$, respectively. Then both $[y] \cup [y^2]$ and $[z]$ are cliques of size $4$ in $\mathcal{S}(D_{2n})$. Moreover, no vertex in $[y] \cup [y^2]$ is adjacent to any vertex in $[z]$. Then taking $T = D_{2n} \setminus ([y] \cup [y^2] \cup [z])$, $\mathcal{S}(D_{2n}) - T $ is disconnected with two components $\mathcal{S}([y] \cup [y^2]) $ and $\mathcal{S} ([z])$ each containing cycles. Thus $\mathcal{S}(D_{2n})$ is again cyclically separable.
		
		Now we prove the converse. Suppose that (i) or (ii) do not hold. This implies that $n$ is $3$ or a power of $2$. If $n = 3$, then $\mathcal{S}(D_{2n}) = \mathcal{S}(e) \vee \left[\mathcal{S} ( \langle a \rangle^*) + \mathcal{S}(\{b, ab, a^{2}b\}) \right]$. Thus $\mathcal{S}(D_{2n}) - \{e\}$ is disconnected and that $\mathcal{S} ( \langle a \rangle^*) \cong K_2$ and $\mathcal{S}(\{b, ab, a^{2}b\}) \cong K_3$. Hence $\mathcal{S}(D_{2n})$ is not cyclically separable. Whereas, if $n$ is a power of $2$, then $\mathcal{S}(D_{2n})$ is a complete graph and hence not cyclically separable.
		
		Finally, suppose that (iii) does not hold. That is $n = 6$ or $n = 12$. First let $n = 6$. Then $\langle a \rangle$ is cyclic group of order $6$. So, $\mathcal{S}(D_{2n}) = \mathcal{S}(\{e, a, a^5\}) \vee \left[\mathcal{S} (\{a^2, a^4\}) + \mathcal{S}(\{a^3, b, ab, \dots, a^{5}b\}) \right]$. Thus, to make $\mathcal{S}(D_{2n})$ disconnected, we must delete the set $ \{e, a, a^5\} $ of vertices and that  $\mathcal{S}(D_{2n}) - \{e, a, a^5\}$ is a disconnected graph with components $\mathcal{S} (\{a^2, a^4\}) \cong K_2$ and $ \mathcal{S}(\{a^3, b, ab, \dots, a^{5}b\}) \cong K_7$. Hence $\mathcal{S}(D_{2n})$ is not cyclically separable. 
		
		Next let $n = 12$. Then $\langle a \rangle$ is cyclic group of order $12$, and that the vertices in $\{e\} \cup [a]$ are adjacent to every other vertices in $\mathcal{S}(D_{2n})$. Let $A = \{a^6, b, ab, \dots, a^{11}b\}$. That is, $A$ is the set of elements of order two in $ D_{2n} $. Moreover, $[a^2]$, $[a^3]$, and $[a^4]$ are the elements of order $6$, $4$, and $3$ in $D_{2n}$. We can visualize the structure of $\mathcal{S}(D_{2n}) - (\{e\} \cup [a])$ as given below.
		\begin{figure}[h]
			
			\begin{center}
				\begin{tikzpicture}[
					roundnode/.style={circle, draw=black},
					squarednode/.style={rectangle, draw=red!60, fill=red!5, very thick, minimum size=5mm},
					]
					\node[roundnode]      (maintopic)                              {$\mathcal{S}([a^2])$};
					\node[roundnode]        (uppercircle)       [left=of maintopic] {$\mathcal{S}(A)$};
					\node[roundnode]      (rightsquare)       [right=of maintopic] {$\mathcal{S}([a^4])$};
					\node[roundnode]        (lowercircle)       [left=of uppercircle] {$\mathcal{S}([a^3])$};
					
					\draw (uppercircle.east) -- (maintopic.west);
					\draw (maintopic.east) -- (rightsquare.west);
					\draw (lowercircle.east) -- (uppercircle.west);
				\end{tikzpicture}
				\caption{$\mathcal{S}(D_{2n}) - (\{e\} \cup [a])$} \label{fig:M1}
			\end{center}
			
		\end{figure}
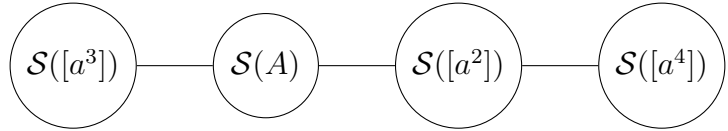

		We observe from the figure that to make $\mathcal{S}(D_{2n}) - (\{e\} \cup [a])$ disconnected, we must delete $A$ or $[a^2]$ or $A \cup [a^2]$. However, since $[a^3]$, and $[a^4]$ are cliques of size two in $\mathcal{S}(D_{2n}) - (\{e\} \cup [a])$, whether we delete $A$ or $[a^2]$ or $A \cup [a^2]$, we will end up getting a disconnected graph with two components and at least one component consists of two vertices. Hence $\mathcal{S}(D_{2n})$ is not cyclically separable.  
	\end{proof}

	\begin{theorem}\label{dicyclic1}
		For any positive integer $n \geq 2$, $\mathcal{S}(Q_{4n})$ is cyclically separable if and only if $n$ is not a power of $2$.
	\end{theorem}
	
	\begin{proof}
		If $n$ is a power of $2$, then $\mathcal{S}(Q_{4n})$ is clearly not cyclically separable. 
		
		Now suppose that $n$ is not a power of $2$. Then $n \geq 3$ and is divisible by some prime $p \geq 3$. Let $x$ be an element of order $2p$ in $Q_{4n}$. Then $[x^2]$ is the set of elements of order $p$. Thus $\{x\} \cup [x^2]$ is a clique of size $p$ in $Q_{4n}$. Let $A = \{b, ab, a^2b, \dots, a^{2n-1}b\}$. Then each element in $A$ has order $4$, and thus it is clique of size $2n$. Then for $S = Q_{4n} \setminus (\{x\} \cup [x^2] \cup A)$, the subgraph $\mathcal{S}(Q_{4n}) - S$ is disconnected with components induced by $\{x\} \cup [x^2] $ and $ A$. Since both are cliques of size at least $3$, $\mathcal{S}(Q_{4n})$ is cyclically separable.
	\end{proof}

We now consider EPPO groups for our study of cyclic separability of order supergraphs.
	
	\begin{theorem}\label{thm.eppo}
		Let $G$ be a EPPO group. Then $\mathcal{S}(G)$ is cyclically separable if and only if $pq $ divides $ |G|$ for some primes $p > q \geq 5$ or at least two of the following three conditions hold:
		\begin{enumerate}[\rm(i)]
			\item $p $ divides $ |G|$ for some prime $p \geq 5$,
			\item $G$ has a Sylow $3$-subgroup which is not of order $3$ or not normal,
			\item $G$ has a Sylow $2$-subgroup which is not of order $2$ or not normal.
		\end{enumerate}
	\end{theorem}
	
	\begin{proof}
		We first assume that $pq $ divides $ |G|$ for some primes $p > q \geq 5$. Then $G$ has elements of $x$ and $y$ of order $p$ and $q$. Clearly, no element of $[x]$ is adjacent to any element of $[y]$. Then for $S := G \setminus ([x] \cup [y])$, the graph $\mathcal{S}(G) \setminus S$ is disconnected and has two components of size at least $\phi(5) = 4$ induced by $[x]$ and $[y]$. Hence $\mathcal{S}(G)$ is cyclically separable.
		
		If $p $ divides $ |G|$ for some prime $p \geq 5$, then $G$ has an element $a$ of order $p$, and so $|[a]|\geq 4$. Next, let $q \in \{2,3\}$ and that $H$ be a Sylow $q$-subgroup of $G$. If $|H| \neq q$, then $H$ is a subgroup of order at least $q^2$. Thus $H^*$ is a clique of size at least $q^2-1 \geq 3$ in $\mathcal{S}(G)$. Whereas, if $H$ is not normal, then by Sylow's theorem, there will be at least $q$ more Sylow $q$-subgroups. Thus if $S_q$ is the set union of all Sylow $q$-subgroups of $G$, then $S^*_q$ is a clique of size at least $(q+1)(q-1) \geq 3$. Hence, if at least two of (i), (ii), and (iii) hold,  then $\mathcal{S}(G)$ is cyclically separable.
		
		Conversely, let $\mathcal{S}(G)$ be cyclically separable. Then the order of $G$ is not a prime power. Suppose that there is no primes $p > q \geq 5$ such that $pq $ divides $ |G|$. Then $|G| = 2^\alpha 3^\beta p^\gamma$ for some prime $p \geq 5$, and integers $\alpha, \beta, \gamma$, at least two of these are positive. 
		
		\smallskip
		\noindent
		\textbf{Case 1.} $|G| = 2^\alpha 3^\beta p^\gamma$, $\gamma \neq 0$. Then (i) holds. From above, $2 $ divides $ |G|$ or $3 $ divides $ |G|$. Thus we get the following subcases.
		
		\smallskip
		\noindent
		\textbf{Subcase 1.} Either $\alpha = 0$ or $\beta = 0$. Let $q \in \{2,3\}$ and if possible, let $G$ have a Sylow $q$-subgroup of $G$ which is of order $q$ and normal. Then $H \cong \mathbb{Z}_q$. Let $\Gamma_1$ and $\Gamma_2$ are the subgraphs of $\mathcal{S}(G)$ induced by the set of elements of order $q$ and by the set of elements whose order is some power of $p$, respectively. Then $\mathcal{S}(G) = \mathcal{S}(\{e\}) \vee ( \Gamma_1 + \Gamma_2 )$. The graph $\mathcal{S}(G)$ can be visualized as follows. 
		
		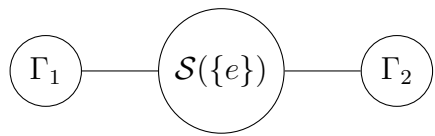
\begin{figure}[h]
			
			\begin{center}
				\begin{tikzpicture}[
					roundnode/.style={circle, draw=black},
					squarednode/.style={rectangle, draw=red!60, fill=red!5, very thick, minimum size=5mm},
					]
					\node[roundnode]      (maintopic)                              {$\Gamma_2$};
					\node[roundnode]        (uppercircle)       [left=of maintopic] {$\mathcal{S}(\{e\})$};
					\node[roundnode]        (lowercircle)       [left=of uppercircle] {$\Gamma_1$};
					
					\draw (uppercircle.east) -- (maintopic.west);
					\draw (lowercircle.east) -- (uppercircle.west);
				\end{tikzpicture}
				\caption{$\mathcal{S}(G)$} \label{fig:M1}
			\end{center}
			
		\end{figure} 
		
		We have $ \mathcal{S}(\{e\}) \cong K_1 $, and  $\Gamma_1 \cong K_1$ if $q  =2$ and $\Gamma_1 \cong K_2$ if $q  =3$. This contradicts the fact that $\mathcal{S}(G)$ is cyclically separable. Hence either (ii) or (iii) hold.
		
		\smallskip
		\noindent
		\textbf{Subcase 2:} $\alpha \beta \neq 0$. If possible, suppose that both the Sylow $2$-subgroups and the $3$-subgroup of $G$ are normal, and are of order $2$ and $3$, respectively.  Let $\Gamma_1$, $\Gamma_2$, and $\Gamma_3$ are the subgroups of $\mathcal{S}(G)$ induced by the set of elements of order $2$ and $3$, and by the set of elements whose order is some power of $p$, respectively. Then $\mathcal{S}(G) = \mathcal{S}(\{e\}) \vee ( \Gamma_1 + \Gamma_2 + \Gamma_3 )$. The graph $\mathcal{S}(G)$ can be visualized as follows. 
		
		\begin{figure}[h]
			
			\begin{center}
				\begin{tikzpicture}[
					roundnode/.style={circle, draw=black},
					squarednode/.style={rectangle, draw=red!60, fill=red!5, very thick, minimum size=5mm},
					]
					\node[roundnode]      (uppergraph)           [above=of uppercircle] {$\Gamma_3$};
					\node[roundnode]      (maintopic)                              {$\Gamma_2$};
					\node[roundnode]        (uppercircle)       [left=of maintopic] {$\mathcal{S}(\{e\})$};
					\node[roundnode]        (lowercircle)       [left=of uppercircle] {$\Gamma_1$};
					
					\draw (uppercircle.north) -- (uppergraph.south);
					\draw (uppercircle.east) -- (maintopic.west);
					\draw (lowercircle.east) -- (uppercircle.west);
				\end{tikzpicture}
				\caption{$\mathcal{S}(G)$} \label{fig:M1}
			\end{center}
			
		\end{figure}
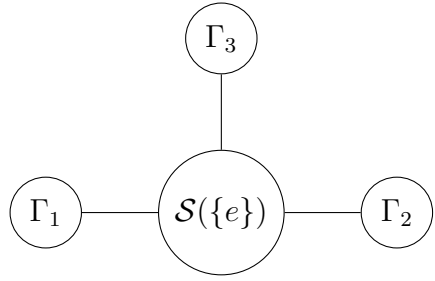 
		
		We have $ \mathcal{S}(\{e\}) \cong K_1 $,  $\Gamma_1 \cong K_1$ and $\Gamma_2 \cong K_2$. This contradicts the fact that $\mathcal{S}(G)$ is cyclically separable. Hence, at least (ii) or (iii) hold.
		
		\smallskip
		\noindent
		\textbf{Case 2:} $|G| = 2^\alpha 3^\beta$, $\alpha \beta \neq 0$. Let $q \in \{2,3\}$ and if possible, let $G$ has a Sylow $q$-subgroup of $G$ which is of order $q$ and normal. Then $H \cong \mathbb{Z}_q$. Let $\Gamma_1$ and $\Gamma_2$ are the subgroups of $\mathcal{S}(G)$ induced by the set of elements whose order is some power of $2$ and $3$, respectively. Then $\mathcal{S}(G) = \mathcal{S}(\{e\}) \vee ( \Gamma_1 + \Gamma_2 )$. We have $ \mathcal{S}(\{e\}) \cong K_1 $, and  $\Gamma_1 \cong K_1$ if $q  =2$ and $\Gamma_2 \cong K_2$ if $q  =3$. This implies that $\mathcal{S}(G)$ is not cyclically separable. As this is a contradiction, both (ii) and (iii) hold.
	\end{proof}
	As a consequence of Theorem \ref{thm.eppo}, we have the following corollary about EPO groups.	
	
	\begin{corollary}
		Let $G$ be an EPO group. Then $\mathcal{S}(G)$ is cyclically separable if and only if $pq $ divides $ |G|$ for some primes $p > q \geq 5$ or at least two of the following three conditions hold:
		\begin{enumerate}[\rm(i)]
			\item $p $ divides $ |G|$ for some prime $p \geq 5$,
			\item $G$ has a Sylow $3$-subgroup which is not cyclic or not normal,
			\item $G$ has a Sylow $2$-subgroup which is not cyclic or not normal.
		\end{enumerate}   
	\end{corollary}

We next classify the finite nilpotent groups with cyclically separable order supergraphs.	
	
	\begin{theorem}
		Let $G$ be a finite nilpotent group. Then $\mathcal{S}(G)$ is cyclically separable if and only if either $|G|$ has at least three prime factors or $|G|$ has exactly two prime factors and at least one of the following conditions holds:
		\begin{enumerate}[\rm(i)]
			\item $pq $ divides $ |G|$ for some primes $p>q\geq 5$,
			\item $G$ has a Sylow $p$-subgroup of exponent at least $p^2$ for some prime $p \geq 5$, and a Sylow $q$-subgroup, where $q \in \{2,3\}$,
			\item $G$ has a Sylow $p$-subgroup of exponent $p$ for some prime $p \geq 5$, and a Sylow $q$-subgroup which is not of order $q$ or not normal, where $q\in  \{2,3\}$,
			\item $G$ has a Sylow $2$-subgroup which is not of order $2$ or not normal, and $G$ has a Sylow $3$-subgroup which is not of order $3$ or not normal,
			\item $G$ has a Sylow $2$-subgroup which is of exponent at least $4$ and not normal, and a Sylow $3$-subgroup,
			\item $G$ has a Sylow $2$-subgroup of exponent at least $8$ and a Sylow $3$-subgroup,
			\item $G$ has a Sylow $3$-subgroup of exponent at least $9$ and a Sylow $2$-subgroup.
		\end{enumerate}   
	\end{theorem}
	
	\begin{proof}
		First, suppose that $|G|$ has at least three prime factors, say $p_1 > p_2 > p_3$. Then $p_1 \geq 5$, $p_2 \geq 3$, $p_3 \geq 2$. Let $a$, $b$, and $c$ be elements of order $p_1 $, $ p_2 $, and $ p_3$, respectively. Then $[a]$ and $[bc] \cup [c]$ are cliques of size $\phi(p_1) =p_1-1 \geq 4$ and $\phi(p_2p_3) + \phi(p_3) =(p_2-1)(p_3-1) + (p_3-1) \geq 2 +1 = 3$, respectively. Hence $G \setminus ([a] \cup [bc] \cup [c])$ is a cyclic cutset of $\mathcal{S}(G)$, and so $\mathcal{S}(G)$ is cyclically separable.
		
		Next, suppose that $|G|$ has exactly two prime factors. Let $p>q\geq 5$ and suppose $G$ has elements $a$ and $b$ of order $p$ and $q$, respectively.  Then $[a]$ and $[b]$ are cliques of size $p-1 \geq 6$ and $q-1 \geq 4$, respectively. Hence $G \setminus ([a] \cup [b])$ is a cyclic cutset of $\mathcal{S}(G)$.
		
		Now let $G$ have a Sylow $p$-subgroup for some prime $p \geq 5$, and a Sylow $q$-subgroup, where $q \in \{2,3\}$. If the Sylow $p$-subgroups have exponent at least $p^2$, then $G$ has an element $a$ of order $p^2$. So $|[a]| = \phi(p^2)\geq 20$. As $G$ also has a Sylow $q$-subgroup, it has an element $b\in G$ of order $q$. Then the order of $a^2b$ is $pq$, and that $|[a^2b]| \geq \phi(pq) \geq 4$. Additionally, none of $\circ(a)|\circ(a^2b)$ and $\circ(a^2b)|\circ(a)$ hold. Hence $G\setminus([a]\cup[a^2b])$ is cyclic cutset of $\mathcal{S}(G)$. Next, let $G$ have a Sylow $p$-subgroup of exponent $p$, and a Sylow $q$-subgroup which is not of order $q$ or not normal. Then from the proof of Theorem \ref{thm.eppo}, $G$ has a cyclic cutset.
		
		If $G$ has a Sylow $2$-subgroup which is not of order $2$ or not normal, and $G$ has a Sylow $3$-subgroup is not of order $3$ or not normal, then again from the proof of Theorem \ref{thm.eppo} we know that $G$ has a cyclic cutset.
		
		Suppose $G$ has a Sylow $2$-subgroup which is of exponent at least $4$ and not normal, and a Sylow $3$-subgroup. Let $a$ and $b$ be elements of order $4$ belonging to different Sylow $2$-subgroups, and that $c$ be an element of order $3$. Then $[a] \cup [b]$ is a clique in $\mathcal{S}(G)$ and $|[a] \cup [b]| = \phi(4) + \phi(4) =4$. Whereas, $a^2c$ is of order $6$, so that $[a^2c] \cup [c]$ is a clique of size $\phi(6) + \phi(3) =4$. Additionally, the order of no element of $[a] \cup [b]$ divides that of any element of $[a^2c] \cup [c]$, and vice versa. Hence, $G \setminus ([a] \cup [b] \cup [a^2c] \cup [c])$ is a cyclic cutset of $\mathcal{S}(G)$.
		
		Next, let $G$ have a Sylow $2$-subgroup of exponent at least $8$ and a Sylow $3$-subgroup. Then there exists an element $a$ of order $8$ and an element $b$ of order $3$ in $G$. Then $a^2b$ is of order $12$. Thus none of $\circ(a) | \circ(a^2b)$ and $\circ{(a^2b)} | \circ{(b)}$ hold. Moreover, $|[a]| = \phi(8)=4$ and $|[a^2b]| = \phi(12) = 4$. Thus,  $G\setminus ([a]\cup [a^2b])$ is a cyclic cutset of $\mathcal{S}(G)$. 
		
		Finally, suppose that $G$ has a Sylow $3$-subgroup of exponent at least $9$ and a Sylow $2$-subgroup. Then there exists an element $a$ of order $9$ and an element $b$ of order $2$ in $G$. So $a^2b$ is of order $6$. Then $[a]$ and $[a^2b] \cup [b]$ are cliques of size $\phi(9) =6$ and $\phi(6) + \phi(2) = 2 +1 = 3$, respectively. Additionally, the order of no element of $[a]$ divides the order of any element of $[a^2b] \cup [b]$ and vice versa. Hence $G \setminus ([a] \cup [a^2b] \cup [b])$ is a cyclic cutset of $\mathcal{S}(G)$.
		
		Therefore, if at least one of (i)-(vi) holds, then $\mathcal{S}(G)$ is cyclically separable.
		
		To prove the converse, let $|G|$ have at most two prime factors and that none of (i)-(vi) holds. If $|G|$ has exactly one prime factor, then $\mathcal{S}(G)$ is complete, and so it is not cyclically separable. Thus $|G|$ have exactly two distinct prime factors. As (i) does not hold, $|G|$ has at most one prime divisor $p \geq 5$.
		
		\noindent
		\textbf{Case 1.} $p $ divides $ |G|$ for some prime $p \geq 5$. Then $G$ has a Sylow $p$-subgroup of exponent $p$, and a Sylow $q$-subgroup which is of order $q$ or normal, where $q\in  \{2,3\}$. If $\Gamma_1$, $\Gamma_2$, and $\Gamma_3$ are subgraphs of $\mathcal{S}(G)$ induced by elements of order $p$, $q$, and $pq$, respectively, then $\mathcal{S}(G) = \mathcal{S}(\{e\}) \vee \left[ \Gamma_3 \vee (\Gamma_1 + \Gamma_2)\right]$. In fact, $\mathcal{S}(G) - \{e\}$ can be visualized as follows.

		\begin{figure}[h]
			
			\begin{center}
				\begin{tikzpicture}[
					roundnode/.style={circle, draw=black},
					squarednode/.style={rectangle, draw=red!60, fill=red!5, very thick, minimum size=5mm},
					]
					\node[roundnode]      (maintopic)                              {$\Gamma_2$};
					\node[roundnode]        (uppercircle)       [left=of maintopic] {$\Gamma_3$};
					\node[roundnode]        (lowercircle)       [left=of uppercircle] {$\Gamma_1$};
					
					\draw (uppercircle.east) -- (maintopic.west);
					\draw (lowercircle.east) -- (uppercircle.west);
				\end{tikzpicture}
				\caption{$\mathcal{S}(G) - \{e\}$} \label{fig:M1}
			\end{center}
			
		\end{figure}
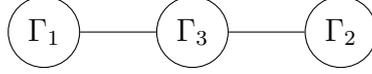 
	
	Let $S$ be the set of elements of order $pq$ in $G$. We observe in the figure that to disconnect $\mathcal{S}(G)$, we must delete $S \cup \{e\}$. However, as $q \in \{2,3\}$, we have $\Gamma_2 \cong K_1$ or $\Gamma_2 \cong K_2$. Thus $\mathcal{S}(G) - (S \cup \{e\})$ have exactly two components and one of them does not contain a cycle.  Hence $\mathcal{S}(G)$ is not cyclically separable.
	
	\smallskip
	\noindent
	\textbf{Case 2.} $p $ divides $ |G|$ for no primes $p \geq 5$. So the prime factors of $|G|$ are $2$ and $3$. As (iv) does not hold, the Sylow $2$-subgroup is of order $2$ and normal or the Sylow $3$-subgroup is of order $3$ and normal. Moreover, as (v) and (vi) do not hold, the Sylow $2$-subgroups are of exponent $2$ or they have exponent $4$ and are normal, and the Sylow $3$-subgroups are of exponent $3$.
	
	\smallskip
\noindent
\textbf{Subcase 1.}	The Sylow $2$-subgroup is of order $2$ and normal. If $\Gamma_1$, $\Gamma_2$, and $\Gamma_3$ are subgraphs of $\mathcal{S}(G)$ induced by elements of order $2$, $3$, and $6$, respectively, then $\mathcal{S}(G) = \mathcal{S}(\{e\}) \vee \left[ \Gamma_3 \vee (\Gamma_1 + \Gamma_2)\right]$. Let $S$ be the set of elements of order $6$ in $G$. Then by argument similar to that of Case 1, to disconnect $\mathcal{S}(G)$, we must delete $S \cup \{e\}$. However, $\Gamma_1 \cong K_1$. Thus $\mathcal{S}(G) - (S \cup \{e\})$ have exactly two components $\Gamma_1$ and $\Gamma_2$, and $\Gamma_1$ does not contain any cycle. Hence $\mathcal{S}(G)$ is not cyclically separable.

	\smallskip
\noindent
\textbf{Subcase 2.}	The Sylow $3$-subgroup is of order $3$ and normal. If the Sylow $2$-subgroups are of exponent $2$, and $\Gamma_1$, $\Gamma_2$, and $\Gamma_3$ are subgraphs of $\mathcal{S}(G)$ induced by elements of order $2$, $3$, and $6$, respectively, then $\mathcal{S}(G) = \mathcal{S}(\{e\}) \vee \left[ \Gamma_3 \vee (\Gamma_1 + \Gamma_2)\right]$. As $\Gamma_2 \cong K_2$, by arguments similar to that of Subcase 1, $\mathcal{S}(G)$ is not cyclically separable. Now let the Sylow $2$-subgroup is of exponent $4$ and normal.
Now, let the Sylow $2$-subgroup have exponent $4$. The exponent of $G$ is 12. Let $X_k$ denote the set of elements of order $k$ in $G$. Then $\mathcal{S}(G) = \mathcal{S} \left(\{e\} \cup X_{12} \right) \vee \left[ \mathcal{S}(X_2) \cup \mathcal{S}(X_3) \cup \mathcal{S}(X_4) \cup \mathcal{S}(X_6) \right]$. Thus to disconnect $\mathcal{S}(G)$, we must delete $\{e\} \cup X_{12}$. 
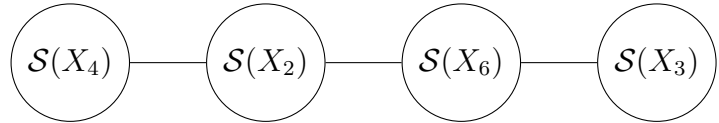
\begin{figure}[h]
	
	\begin{center}
			\begin{tikzpicture}[
					roundnode/.style={circle, draw=black},
					squarednode/.style={rectangle, draw=red!60, fill=red!5, very thick, minimum size=5mm},
					]
					\node[roundnode]      (maintopic)                              {$\mathcal{S}(X_6)$};
					\node[roundnode]        (uppercircle)       [left=of maintopic] {$\mathcal{S}(X_2)$};
					\node[roundnode]      (rightsquare)       [right=of maintopic] {$\mathcal{S}(X_3)$};
					\node[roundnode]        (lowercircle)       [left=of uppercircle] {$\mathcal{S}(X_4)$};
					
					\draw (uppercircle.east) -- (maintopic.west);
					\draw (maintopic.east) -- (rightsquare.west);
					\draw (lowercircle.east) -- (uppercircle.west);
				\end{tikzpicture}
			\caption{$\mathcal{S}(G) - (\{e\} \cup X_{12})$} \label{fig:M1}
		\end{center}
\end{figure} 

We observe that $ \mathcal{S}(X_2) $, $ \mathcal{S}(X_3) $, $ \mathcal{S}(X_4) $, and $ \mathcal{S}(X_6)  $ are cliques of size 1, 2, 2, and 2, respectively. Next, to make $\mathcal{S}(G)$ disconnected, we must delete $X_2$ or $ X_6 $ or $X_2 \cup X_6$ from $\mathcal{S}(G) - (\{e\} \cup X_{12})$.
If we delete $X_2$, then $ \mathcal{S}(X_4) $ and $ \mathcal{S}(X_3) \cup \mathcal{S}(X_6) $ are the two components of $\mathcal{S}(G) - (\{e\} \cup X_{12})$. Whereas, if we delete $X_6$, then $ \mathcal{S}(X_2) \cup  \mathcal{S}(X_4) $ and $ \mathcal{S}(X_3) $ are the two components of $\mathcal{S}(G) - (\{e\} \cup X_{12})$. However, since $|X_4| = |X_6| = 2$ and neither $ \mathcal{S}(X_2) $ nor $ \mathcal{S}(X_6) $ is a cyclic cutset of $\mathcal{S}(G) - (\{e\} \cup X_{12})$. Hence $\mathcal{S}(G)$ is not cyclically separable.
	\end{proof}

We finally investigate the order supergraphs of the symmetric and alternating groups. 	
	
	\begin{theorem}\label{thm.sym}
		For any positive integer $n$,  $\mathcal{S}(S_n)$ is cyclically separable if and only if $n\geq 4$.
	\end{theorem}
	\begin{proof}
		An element $\mu$ of $S_n$ is said to be of type $(1^{m_1}, 2^{m_2}, \cdots, l^{m_l})$, if for $1 \leq i \leq m$, $\mu$ has $m_i$ many $i$-cycles . We know the number of elements in the conjugacy  class represented by $\mu$ is
		$${\frac{n!}{\prod_{r}r^{m_r}{m_r}!}},$$
		where $r$ denotes the length of a cycle and $m_r$ denotes the occurrence of the cycles of length $r$. Now, $S_2=\{e,(1,2)\}$ and $S_3=\{e,(1,2), (1,3),(2,3),(1,2,3),(1,3,2)\}$, which are not cyclically separable. For $n\geq 4$, the number of element of cycles of type $(n^1)$ in $S_n$ is $\frac{n!}{n. 1!}=n-1\geq 3$, and the number of cycles of type $((n-1)^1)$ is $\frac{n!}{(n-1).1!}=n.(n-2)!\geq 8$. It is clear that $n-1$ does not divide $n$ unless $n=2$. Now let, $[\alpha]$ denote the set of all elements in $S_n$ of cycle type $(n^1)$ and $[\beta]$ denote the set of all elements in $S_n$ of cycle type $(1^1,(n-1)^1)$. Then for $T:= S_n \setminus([\alpha]\cup[\beta])$, the graph $\mathcal{S}(S_n)\setminus T$ is disconnected and has two components each of size $\geq 3$ induced by $[\alpha]$ and $[\beta]$. Hence, $\mathcal{S}(S_n)$ is cyclically separable if and only if $n\geq 4$.
	\end{proof}
	
	\begin{theorem}
		For any positive integer $n$,  $\mathcal{S}(A_n)$ is cyclically separable if and only if $n\geq 4$.
	\end{theorem}
	\begin{proof}
		Since $A_3=\{e,(1,2,3),(1,3,2)\}$, so it is not cyclically separable. Now, for $n\geq 4$ the number of cycles of the type $(a,b)(c,d)$ is $\frac{n!}{(2^2. 2!)(1^{n-4}(n-4)!)}=\frac{n(n-1)(n-2)(n-3)}{2^2.2}\geq 3$.
		
		\noindent {\bf Case 1.}  $n$ is even. Then $n-1$ is odd. If $\mu$ is an element of cycle type $(1^1,(n-1)^1)$ that is $\mu$ is a $(n-1)$ cycle, then $\mu \in A_n$. Now, from Theorem \ref{thm.sym}, we know for $n\geq 4$ the number of cycles of type $(1^1,(n-1)^1)$ in $S_n$ is $\geq 8$. That is, for $n\geq 4$ the number of cycles of type $(1^1,(n-1)^1)$ in $A_n$ is $\geq 8$. Let, $[\alpha]$ denote the set of all elements in $A_n$ of cycle type $(1^0,2^2,3^0, \cdots ,n^0)$ and $[\beta]$ denote the set of all elements in $A_n$ of cycle type $(1^1,(n-1)^1)$. Then for $T:= A_n \setminus([\alpha]\cup[\beta])$, the graph $\mathcal{S}(A_n)\setminus T$ is disconnected and has two components each of size $\geq 3$ induced by $[\alpha]$ and $[\beta]$. Hence, $\mathcal{S}(A_n)$ is cyclically separable if and only if $n\geq 4$.
		
		\noindent {\bf Case 2.}  $n$ is odd. In this case the proof will be similar to the proof of Case 1. But here $[\beta]$ denotes the set of all elements in $A_n$ of cycle type $(n^1)$.
	\end{proof}

\end{document}